\documentclass[12pt]{article}
\input{epsf.sty}
\usepackage{epsfig}
\usepackage{amsmath}
\usepackage{latexsym}
\usepackage{amssymb}
\usepackage{bm}
\usepackage{cite}
\usepackage{amsthm}
\theoremstyle{definition}
\newtheorem{theorem}{Theorem}
\newtheorem{lemma}{Lemma}

\newtheorem{assumption}{Assumption}
\newtheorem{proposition}{Proposition}
\newtheorem{corollary}{Corollary}
\newtheorem{definition}{Definition}

\newtheorem{remark}{Remark}

 \graphicspath{/figures}
 \DeclareMathOperator{\sign}{sign}

\makeatletter
\let\normalequation=\equation
\def\equation{\@ifnextchar[{\subequation}{\normalequation}}
\def\subequation[#1]#2{\@ifundefined{r@#1}%
  {\def\theequation{\bf ??#2}\@warning
    {Reference `#1' on page \thepage \space
     undefined}}{\edef\@tempa{\@nameuse{r@#1}}%
    \edef\theequation{\expandafter\@car\@tempa \@nil#2}}%
  \let\@currentlabel\theequation $$}
\makeatother

\begin{document}
\begin{titlepage}
\begin{center}
{\large\bf New criterion of asymptotic stability for delay systems with
time-varying structures and delays}\footnote{This work is jointly supported
by the National Natural Sciences Foundation of China under Grant Nos.
61273211 and 61273309, the Foundation for the Author of National Excellent
Doctoral Dissertation of P.R. China No. 200921, and China Postdoctoral
Science Foundation No. 2011M500065, the Marie Curie International Incoming
Fellowship from the European Commission (FP7-PEOPLE-2011-IIF-302421).}
\\[0.2in]
\begin{center}
Bo Liu\footnote{Bo Liu is with Key Laboratory of Intelligent Perception and
Image Understanding of Ministry of Education of China, Xidian University,
Xi'an
 710071, P.R. China, and Institute of
Industrial Science, The University of Tokyo, 4-6-1 Komaba, Meguro-ku, Tokyo
153-8505, Japan (liu7bo9@gmail.com). }, Wenlian Lu\footnote{Wenlian Lu is
with Centre for Computational Systems Biology, and Laboratory of
Mathematics for Nonlinear Science, School of Mathematical Sciences, Fudan
University, Shanghai, People's Republic of China (wenlian@fudan.edu.cn).},
Tianping Chen\footnote{Tianping Chen is with the School of
Computer/Mathematical
Sciences, Fudan University, 200433, Shanghai, China. \\
\indent ~~Corresponding author: Tianping Chen. Email: tchen@fudan.edu.cn}
\end{center}
\end{center}

\begin{abstract}
In this paper, we study asymptotic stability of the zero solution of a
class of differential systems governed by a scalar differential inequality
with time-varying structures and delays. We establish a new generalized
Halanay inequality for the asymptotic stability of the zero solution for
such systems under more relaxed conditions than the existing ones. We also
apply the theoretical results to the analysis of self synchronization in
networks of delayed differential systems and obtained a more general
sufficient condition for self synchronization.
\end{abstract}


\end{titlepage}

\pagestyle{plain}

\section{Introduction}\quad
The applications of delay differential equations can be found in
many areas including control systems, neural networks, and many
others. And a fundamental problem in these applications is to
determine the stability of the solutions, which has been analyzed
for decades. For example, in order to analyze the asymptotic
stability of the zero solution of the following delay-differential
equations with fixed delay $\tau>0$,
\begin{eqnarray}
  \dot{x}(t)=-ax(t)+bx(t-\tau),
\end{eqnarray}
Halanay(1966) proved the following inequality which was later called
{\em Halanay inequality}.
\begin{proposition}\label{propOrigin}
Let $x(t)>0$, $t\in\mathbb{R}$, be a differentiable scalar function
of $t$ that satisfies
\begin{eqnarray}
  \dot{x}(t)\le -ax(t)+b\sup_{t-\tau\le s\le t}x(s),&t\ge t_{0}\\
  x(t)=\psi(t),&t\le t_{0}
\end{eqnarray}
with $a>b>0$ being constants and $\psi(t)\ge 0$ continuous and
bounded for $t\le t_{0}$, then there exist $k>0$ and $\gamma>0$ such
that $x(t)\le ke^{-\gamma(t-t_{0})}$. Hence $x(t)\to 0$ as
$t\to\infty$.
\end{proposition}

Later on, this inequality has been extended to more general types of
delay differential equations. For example, in Baker \& Tang
(1996),Wen, Yu \& Wang (2008), it has been proved
\begin{proposition}\label{propGeneralOne}
Let $x(t)>0$, $t\in \mathbb{R}$ be a differentiable scalar function
that satisfies
  \begin{eqnarray}
    \dot{x}(t)\le -a(t)x(t)+b(t)\sup_{q(t)\le s\le t}x(s),&t\ge
    t_{0},\\
x(t)=\psi(t),&t\le t_{0},
  \end{eqnarray}
  where $\psi(t)>0$ is bounded and continuous for $t\le t_{0}$, $a(t)$, $b(t)\ge 0$
  for $t\ge t_{0}$, $0<q(t)\le t$ and $q(t)\to \infty$ as
  $t\to\infty$.
If there exists $\sigma>0$ such that
  \begin{align}\label{eqnConditionOld}
    -a(t)+b(t)\le -\sigma<0, t\ge t_{0},
  \end{align}
  then (i) $x(t)\le \sup_{-\infty< s\le t_{0}}|\psi(s)|$, (ii) $x(t)\to 0$
  as $t\to\infty$.
\end{proposition}

In Mohamad \& Gopalsamy (2000), the authors consider continuous and
discrete time Halanay-type inequalities and further generalize the results
of Baker \& Tang (1996) to the case of distributed delays.
\begin{proposition}(Theorem 2.2 of Mohamad \& Gopalsamy, 2000)
Let $x(t)$, $t\in\mathbb{R}$ be a nonnegative function that
satisfies
\begin{eqnarray*}
\dot{x}(t)\le -a(t)x(t)+b(t)\int_{0}^{\infty}K(s)x(t-s)ds,& t>t_{0},\\
x(t)=|\varphi(t)| & t\le t_{0},
\end{eqnarray*}
where $\varphi(s)$, $s\in (-\infty,t_{0}]$, $a(t)$ and $b(t)$, $t\in
\mathbb{R}$ are nonnegative continuous and bounded functions; the
delay kernel $K(\cdot) : [0,\infty)\to [0,\infty)$ satisfies
\begin{eqnarray*}
\int_{0}^{\infty}K(s)e^{\alpha s}ds<\infty,
\end{eqnarray*}
for some positive number $\alpha$. Suppose further that
\begin{eqnarray}\label{eqnConditionDistributedDelay}
a(t)-b(t)\int_{0}^{\infty}K(s)ds\ge \sigma,~~t\in \mathbb{R},
\end{eqnarray}
where $\sigma=\inf_{t\in
\mathbb{R}}[a(t)-b(t)\int_{0}^{\infty}K(s)ds]>0$. Then there exists
a positive number $\tilde{\alpha}$ such that
\begin{eqnarray*}
x(t)\le \Big(\sup_{s\le
t_{0}}x(s)\Big)e^{-\tilde{\alpha}(t-t_{0})},~~t>t_{0}.
\end{eqnarray*}
\end{proposition}

In Baker (2010), the author made some refinement on the decay rate of their
pervious works.

Generalized Halanay inequalities have also been developed in the
stability analysis of delay differential systems. For example, in
Chen (2001), Lu \& Chen (2004), the authors proposed some variants
of the Halanay inequality to solve the global stability of delayed
Hopfield neural networks.

Particularly, in Chen \& Lu (2003), Lu \& Chen (2004), the following
periodic and almost periodic integro-differential systems
\begin{eqnarray}
&&\frac{du_{i}(t)}{dt}=-d_{i}(t)u_{i}(t)+\sum_{j=1}^{n}a_{ij}(t)g_j(u_j(t))
+\sum_{j=1}^{n}\int_{0}^{\infty}f_{j}(u_{j}(t-\tau_{ij}(t)-s))d_{s}K_{ij}(t,s)+I_i(t),
\nonumber\\&&i=1,2,\ldots,n,\label{nnge}
\end{eqnarray}
where $d_{s}K_{ij}(t,s)$ are Lebesgue-Stieltjes measures for each
$t$, are discussed.

As a direct consequence of the main Theorem in Lu \& Chen (2004), we
have

\begin{proposition} {\it Suppose that $|g_{j}(x)|\le G_{j}|x|$ and $|f_{j}(x)|\le
F_{j}|x|$. If there exist positive constants $\xi_{1},\xi_{2},
\cdots,\xi_{n}$, $\alpha$ such that for all $t>0$ and
$i=1,2,\cdots,n$,
\begin{eqnarray}
-\xi_{i}(d_{i}(t)-\alpha)+\sum\limits_{j=1
}^{n}\xi_{j}G_{j}|a_{ij}(t)|
+\sum\limits_{j=1}^{n}\xi_{j}F_{j}e^{\alpha\tau_{ij}(t)}
\int_{0}^{\infty}e^{\alpha s}|d_{s}K _{ij}(t,s)|\le 0,
\end{eqnarray}
then for any solution $u(t)=[u_{1}(t),\cdots,u_{n}(t)]$, $t>0$ of
the system (\ref{nnge}) with $I_{i}(t)=0$, $i=1,\cdot,n$, we have
\begin{eqnarray}
\max_{i=1,\cdots,n}|u_{i}(t)|\le \max_{i=1,\cdots,n}\max_{-\tau\le
s\le 0}(e^{\alpha s }|u_{i}(s)|)e^{-\alpha t}.
\end{eqnarray}}
\end{proposition}

In particular, when $n=1$, $d_{s}K _{11}(t,s)=b(t)\delta(s)$,
$\tau_{11}(t)=\tau(t)$, we have
\begin{proposition}  (also see Chen, 2001) {\it Suppose
$-(a(t)-\alpha)+|b(t)|e^{\alpha \tau(t)} \le 0$, then for any
continuous scalar function $x(t)\ge 0$ that satisfies
\begin{eqnarray}
\left\{\begin{array}{rcl}
\dot{x}(t)&\le& -a(t)x(t)+|b(t)|\sup_{s\ge 0}x(t-s),~~t>0,\\
x(t)&=&|\varphi(t)|, ~~ t\le 0,
\end{array}\right.
\end{eqnarray}
we have
\begin{eqnarray}
|x(t)|\le \max_{-\tau\le s\le 0}(e^{\alpha s }|\phi(s)|)e^{-\alpha
t}.
\end{eqnarray}}
\end{proposition}

Instead, when $n=1$, $d_{s}K _{11}(t,s)=b(t)k(s)ds$,
$\tau_{11}(t)=0$, we have
\begin{proposition} {\it Suppose
$-(a(t)-\alpha)+b(t) \int_{0}^{\infty}e^{\alpha s}K(s)ds\le 0$, then
for any continuous scalar function $x(t)$ satisfying
\begin{eqnarray}
\left\{\begin{array}{rcl}
\dot{x}(t)&\le& -a(t)x(t)+b(t)\int_{0}^{\infty}K(s)x(t-s)ds,~~t>0,\\
x(s)&=&|\varphi(s)|,~~ t\le 0,
\end{array}\right.
\end{eqnarray}
we have
\begin{eqnarray}
|x(t)|\le \max_{-\tau\le s\le 0}(e^{\alpha s }|\phi(s)|)e^{-\alpha
t}.
\end{eqnarray}}
\end{proposition}

For more recent works, refer to Liu, Lu \& Chen (2011) and Gil'
(2013). In all the above mentioned works, there is a basic
requirement: \underline{$a(t)>b(t)$} for all $t$. This requirement
is not satisfied in many real systems. For example, it is well known
that a system switching among several subsystems can be stable even
not all the subsystems are stable. So it is necessary, if possible,
to further generalize the Halanay
inequality so that it can be used to more general cases. 

In this paper, we will first generalize the differential
inequalities with bounded time-varying delays under more relaxed
requirements, say, without $a(t)>b(t)$ for all $t$. Then, we provide
two applications of the theoretical results. First, we apply the
theoretical results to the analysis of self synchronization in
neural networks. Based on our new generalized Halanay inequality, we
proved new sufficient conditions for self synchronization in neural
networks with bounded time-varying delays. Then, we investigate
periodic solutions of neural networks with periodic coefficients and
time delays. Under more relaxed requirement, we proved new
sufficient conditions for the existence and exponential stability of
the periodic solutions of such neural networks.

The rest of the paper is organized as follows. In Section
\ref{secGeneralHalanay}, the new generalized Halanay inequality is
proposed and proved; two applications of the theoretical results are
given in Section \ref{secApplication}; Numerical examples with
simulations are given in Section \ref{secExample}; the paper is
concluded in Section \ref{secConclusion}.

\section{Generalized Halanay inequality}\label{secGeneralHalanay}
Consider a scalar function $x(t)$ governed by the inequality

\begin{eqnarray}\label{ineqnSystem1}\left\{
\begin{array}{rcl}
D^{+}|x(t)|&\le& -a(t)|x(t)|+b(t)\sup_{t-\tau_{\max}\le s\le
t}|x(t-s)|,~~t\ge 0,\\
x(s)&=&\phi(s), s\in [-\tau_{\max},0],
\end{array}\right.
\end{eqnarray}
where $D^{+}$ represents the upper right Dini derivative,
$a(\cdot)$: $\mathbb{R}^{+}\mapsto \mathbb{R}^{+}$, $b(\cdot)$:
$\mathbb{R}^{+}\mapsto \mathbb{R}$ are piecewise continuous and
uniformly bounded, i.e., there exists $M_{a}>0$, $M_{b}>0$ such that
$0<a(t)\le M_{a}$, $|b(t)|\le M_{b}$, $\phi(s)\ge 0$ is the initial
value, and $\tau(\cdot)$: $\mathbb{R}^{+}\mapsto (0,\tau_{\max}]$ is
the time-varying delay with $\tau_{\max}$ being the upper bound.

For a fixed $\eta>0$ and $0\le t_{1}<t_{2}<\infty$, denote the set
$S_{\eta}(t_{1},t_{2})=\{t\in (t_{1},t_{2}):~a(t)-|b(t)|>\eta\}$, and the
set $S_{-}(t_{1},t_{2})=\{t\in (t_{1},t_{2}):~a(t)<|b(t)|\}$,
$S_{+}(t_{1},t_{2})=(t_{1},t_{2})\backslash (S_{\eta}(t_{1},t_{2})\cup
S_{-}(t_{1},t_{2}))$. It is obvious that $S_{\eta}(t_{1},t_{2})$,
$S_{-}(t_{1},t_{2})$, $S_{+}(t_{1},t_{2})$ are composed of a series of
intervals, and $0\le a(t)-|b(t)|\le \eta$ on $S_{+}(t_{1},t_{2})$. Let
$\mu_{s}(t_{1},t_{2})=\mu(S_{s}(t_{1},t_{2}))$ be the Lebesgue measure of
the set $S_{s}(t_{1},t_{2})$, where $s=\eta$, '+', '-'. And the parameter
$\delta=1-\displaystyle\frac{\eta}{2M_{a}}\in (0,1)$ will also be used
later.

Before stating main results, we summarize the basic conditions into the
following

\begin{definition}[$\eta$-condition]\label{conditionBasic}
A function pair $\{a(\cdot),b(\cdot)\}$ with $0<a(\cdot)\le M_{a}$,
$|b(\cdot)|\le M_{b}$ is said to satisfy the {\bf $\eta$-condition},
if there exist $t_{0}\ge 0$, $0<C^{*}<\eta/2$ and an integer $N>0$
such that
\begin{enumerate}
\item[(i)]
\begin{eqnarray}\label{eqnConditionMaina}
\sum_{k=0}^{+\infty}\mu_{\eta}(t_{k},t_{k+1}^{-})=\infty;
\end{eqnarray}
\item[(ii)]
\begin{align}\label{eqnConditionMain}
\limsup_{k\to\infty}\frac{[e^{M_{b}\mu_{-}(t_{k},t_{k+1})}-1]
e^{M_{a}\mu_{+}(N+1)\tau_{\max}}}
{\min\{\frac{1}{M_{a}},\mu_{\eta}(t_{k},t_{k+1}^{-})\}}=C^{*},
\end{align}
\end{enumerate}
where $t_{k}=t_{0}+k(N+1)\tau_{\max}$, and
$t_{k}^{-}=t_{k}-\tau_{\max}$.
\end{definition}
\begin{remark}
If $0/0$ appears on the left-hand side of \eqref{eqnConditionMain},
then we explain it as $0$. Thus \eqref{eqnConditionMain} will always
hold when the numerator on its left-hand side is zero.
\end{remark}

Now, we state the main result which can be called a ``Generalized
Halanay Inequality".

\begin{theorem}\label{thmMain}For any given $a(\cdot)$, $b(\cdot)$ in \eqref{ineqnSystem1}, if there
exists $\eta>0$ such that $\{a(\cdot),b(\cdot)\}$ satisfies the
$\eta$-condition, then the zero solution of any system governed by
\eqref{ineqnSystem1} is asymptotically stable, i.e., from any
initial value $\phi(s)$, $s\in [-\tau_{\max},0]$, there exists $K>1$
such that the solution $x(t)$ satisfies
$$ |x(t)|\le K\max_{-\tau_{\max}\le s\le 0}|\phi(s)|,$$ for all $t\ge 0$, and
$x(t)\to 0$ as $t\to \infty$. Furthermore, if there exists
$\epsilon>0$ such that $\mu_{\eta}(t_{k},t_{k+1})\ge\epsilon$ for
each $k$, then the convergence is exponential, i.e., there exists
$\widetilde{K}>0$, $\alpha>0$ such that
\begin{eqnarray*}
|x(t)|\le \widetilde{K}\max_{-\tau_{\max}\le s\le
0}|\phi(s)|e^{-\alpha t}.
\end{eqnarray*}
\end{theorem}

As a direct consequence of Theorem \ref{thmMain}, we have
\begin{corollary}\label{corMain}
If $\mu_{-}(0,+\infty)=0$ and there exists $\eta>0$ such that
$\mu_{\eta}(0,+\infty)=+\infty$, then the zero solution of any
system governed by \eqref{ineqnSystem1} is asymptotically stable.
\end{corollary}

\begin{remark}
In previous works, it was always assumed that for all $t\ge 0$,
$a(t)-|b(t)|>\eta$. Theorem \ref{thmMain} indicates that even the condition
$a(t)-|b(t)|>\eta$ is not satisfied for some $t>0$, the zero solution of
any system governed by \eqref{ineqnSystem1} can still be asymptotically
stable if only (\ref{eqnConditionMaina}) and (\ref{eqnConditionMain}) are
satisfied. Corollary 1 indicates that in case $a(t)-|b(t)|\ge 0$, only if
the measure of the set satisfying $a(t)-|b(t)|\ge \eta$ is infinite, then
the ``0" is a stable equilibrium point.

\end{remark}

In the proof of Theorem \ref{thmMain}, instead of proving $|x(t)|\to
0$, we prove the following maximal function
\begin{align}
M_{0}(t)=\sup_{t-\tau_{\max}\le s\le t}|x(s)| \end{align} tends to
zero as $t\to \infty$. The basic idea for the proof is to establish
a uniform estimation for $M_{0}(t)$ on an interval $(t_{1},t_{2})$
when $S_{+}(t_{1},t_{2})$, $S_{-}(t_{1},t_{2})$ and
$S_{\eta}(t_{1},t_{2})$ coexist. This will be proved by induction.

Before proving Theorem \ref{thmMain}, we need to make some
preparations by establishing some lammas.

\begin{lemma}\label{lemNonincreaseEstimation1}
$M_{0}(t)$ is nonincreasing on the set $S_{+}(0,+\infty)$ as well as
on the set $S_{\eta}(0,+\infty)$ for any given $\eta>0$.
\end{lemma}
\begin{proof}
Given $t_{1}\in S_{+}(0,+\infty)\cup S_{\eta}(0,+\infty)$, then
$a(t_{1})\ge |b(t_{1})|$. If $|x(t_{1})|<M_{0}(t_{1})$, then from
the continuity of $x(t)$, there exists $t_{2}>t_{1}$ such that
$|x(t)|\le M_{0}(t_{1})$ on $[t_{1},t_{2}]$, which implies that
$M_{0}(t)$ is nonincreasing at $t_{1}$. Otherwise,
$|x(t_{1})|=M_{0}(t_{1})$. We have:
\begin{align*}
\Big\{D^{+}|x(t)|\Big\}_{t=t_{1}}\le
-a(t_{1})|x(t_{1})|+|b(t_{1})|M_{0}(t_{1})
\le-[a(t_{1})-b(t_{1})]|x(t_{1})| \le 0.
\end{align*}
This also implies that $M_{0}(t)$ is nonincreasing at $t_{1}$. The
proof is completed.
\end{proof}

\begin{lemma}\label{lemIncreasingEstimation}
Given any $t_{1}<t_{2}$, we have
\begin{align*}
M_{0}(t_{2})\le M_{0}(t_{1})e^{M_{b}\mu_{-}(t_{1},t_{2})}.
\end{align*}
\end{lemma}
\begin{proof}
Let $t_{1}\le \underline{t}_{1}<\overline{t}_{1}<\underline{t}_{2}
<\overline{t}_{2}<\cdots<\underline{t}_{s}<\overline{t}_{s}\le
t_{2}$ such that
$S_{-}(t_{1},t_{2})=\bigcup_{i=1}^{s}(\underline{t}_{i},\overline{t}_{i})$.
From Lemma \ref{lemNonincreaseEstimation1}, we see that $M_{0}(t)$
can increase only on $S_{-}(t_{1},t_{2})$. Thus,
$M_{0}(\underline{t}_{1})\le M_{0}(t_{1})$. On the other hand, if
$M_{0}(t)$ is increasing at $t^{*}\in
(\underline{t}_{1},\overline{t}_{1})$,
\begin{align*}
\Big\{D^{+}M_{0}(t)\Big\}_{t=t^{*}}=\Big\{D^{+}|x(t)|\Big\}_{t=t^{*}}
\le -a(t)|x(t)|+|b(t)|M_{0}(t)\le M_{b}M_{0}(t).
\end{align*}
This implies that
\begin{eqnarray*}
M_{0}(\overline{t}_{1})\le M_{0}(\underline{t}_{1})e^{M_{b}(\overline{t}_{1}-\underline{t}_{1})}
\le M_{0}(t_{1})e^{M_{b}(\overline{t}_{1}-\underline{t}_{1})}.
\end{eqnarray*}
Similarly, we have
\begin{align*}
M_{0}(\overline{t}_{2})\le
M_{0}(\underline{t}_{2})e^{M_{b}(\overline{t}_{2}-\underline{t}_{2})}
\le
M_{0}(\overline{t}_{1})e^{M_{b}(\overline{t}_{2}-\underline{t}_{2})}
\le
M_{0}(t_{1})e^{M_{b}[(\overline{t}_{1}-\underline{t}_{1})+(\overline{t}_{2}-\underline{t}_{2})]}.
\end{align*}
Repeating this process, finally we can have:
\begin{align*}
M_{0}(t_{2})\le M_{0}(\overline{t}_{s})\le
M_{0}(t_{1})e^{M_{b}\sum_{i=1}^{s}(\overline{t}_{i}-\underline{t}_{i})}
=M_{0}(t_{1})e^{M_{b}\mu_{-}(t_{1},t_{2})}.
\end{align*}
The proof is completed.
\end{proof}

\begin{lemma}\label{lemEstimationMain1}
For any $t_{1}< t_{2}$,
\begin{enumerate}
\item
if $(t_{1},t_{2})=S_{+}(t_{1},t_{2})$, then
\begin{eqnarray}
|x(t_{2})|\le  M_{0}(t_{1})-[M_{0}(t_{1})-|x(t_{1})|]e^{-M_{a}(t_{2}-t_{1})};
\end{eqnarray}

\item
if $(t_{1},t_{2})=S_{\eta}(t_{1},t_{2})$, then for any
$\widetilde{M}\ge M_{0}(t_{1})$,
\begin{align}
|x(t_{2})|&\le \max\{\delta \widetilde{M},|x(t_{1})|
-\frac{\eta}{2}\mu(S_{\eta}(t_{1},t_{2}))\widetilde{M}\};
\end{align}

\item
if $(t_{1},t_{2})=S_{-}(t_{1},t_{2})$, then
\begin{align*}
|x(t_{2})| \le&|x(t_{1})|+M_{0}(t_{1})[e^{M_{b}(t_{2}-t_{1})}-1].
\end{align*}
\end{enumerate}
\end{lemma}
\begin{proof}

\begin{enumerate}
\item $(t_{1},t_{2})=S_{+}(t_{1},t_{2})$; In this case, by Lemma
\ref{lemNonincreaseEstimation1}, $M_{0}(t)$ is nonincreasing on
$(t_{1},t_{2})$. Then,
\begin{align*}
D^{+}|x(t)|\le -a(t)|x(t)|+|b(t)|M_{0}(t_{1}).
\end{align*}
By some calculations, we have
\begin{eqnarray*}
|x(t_{2})|&\le & M_{0}(t_{1})-[M_{0}(t_{1})-|x(t_{1})|]e^{-M_{a}(t_{2}-t_{1})}.
\end{eqnarray*}
\item
$(t_{1},t_{2})=S_{\eta}(t_{1},t_{2})$; From Lemma
\ref{lemNonincreaseEstimation1}, $M_{0}(t)$ is nonincreasing on
$(t_{1},t_{2})$. For any $t\in (t_{1},t_{2})$, if $|x(t)|\ge \delta
\widetilde{M}\ge \delta M_{0}(t)$ for $t\in (t_{1},t_{2})$, then
\begin{align*}
D^{+}|x(t)|\le -a(t)|x(t)|+|b(t)|\widetilde{M} \le -a(t)\delta
\widetilde{M}+|b(t)|\widetilde{M} =-[\delta
a(t)-|b(t)|]\widetilde{M} &\le -\frac{\eta}{2}\widetilde{M}.
\end{align*}
Thus,
\begin{align*}
|x(t_{2})|&\le \max\{\delta \widetilde{M},|x(t_{1})|
-\frac{\eta}{2}\mu(S_{\eta}(t_{1},t_{2}))\widetilde{M})\}\nonumber\\
&=M_{0}(t_{1})-\big[M_{0}(t_{1})-\max\{\delta
M_{0}(t_{1}),|x(t_{1})|
-\frac{\eta}{2}\mu(S_{\eta}(t_{1},t_{2}))M_{0}(t_{1})\}\big].
\end{align*}

\item
$(t_{1},t_{2})=S_{-}(t_{1},t_{2})$; If $M_{0}(t)\le M_{0}(t_{1})$ on
$(t_{1},t_{2})$, then
\begin{eqnarray}
D^{+}|x(t)|\le -a(t)|x(t)|+|b(t)|M_{0}(t_{1})\le M_{b}M_{0}(t_{1}).
\end{eqnarray}
Thus,
\begin{eqnarray*}
  |x(t_{2})|\le |x(t_{1})|+M_{b}M_{0}(t_{1})(t_{2}-t_{1})\le
  |x(t_{1})|+M_{0}(t_{1})\big[e^{M_{b}(t_{2}-t_{1})}-1\big].
\end{eqnarray*}
Otherwise, let $t^{*}=\inf\{t\in
(t_{1},t_{2}):~|x(t^{*})|=M_{0}(t_{1})\}$.

Then for $t\in (t_{1},t^{*})$, we have
\begin{eqnarray}
D^{+}|x(t)|\le -a(t)|x(t)|+|b(t)|M_{0}(t_{1})\le M_{b}M_{0}(t_{1}).
\end{eqnarray}
which implies
\begin{eqnarray*}
M_{0}(t_{1})=|x(t^{*})|\le
|x(t_{1})|+M_{b}M_{0}(t_{1})(t^{*}-t_{1})\le
|x(t_{1})|+M_{0}(t_{1})\big[e^{M_{b}(t^{*}-t_{1})}-1\big].
\end{eqnarray*}

Therefore, (noting Lemma \ref{lemIncreasingEstimation}), we have:
\begin{align*}
|x(t_{2})|&\le M_{0}(t_{2})\le M_{0}(t^*)e^{M_{b}(t_{2}-t^{*})}\le
M_{0}(t_{1})e^{M_{b}(t_{2}-t^{*})}\\
&=M_{0}(t_{1})+ M_{0}(t_{1})\big[e^{M_{b}(t_{2}-t^{*})}-1\big]\\
&\le|x(t_{1})|+M_{0}(t_{1})\big[e^{M_{b}(t^{*}-t_{1})}-1\big]+M_{0}(t_{1})\big[e^{M_{b}(t_{2}-t^{*})}-1\big]\\
&\le|x(t_{1})|+M_{0}(t_{1})\big[e^{M_{b}(t^{*}-t_{1})}-1\big]+M_{0}(t_{1})e^{M_{b}(t^{*}-t_{1})}\big[e^{M_{b}(t_{2}-t^{*})}-1\big]\\
&=|x(t_{1})|+M_{0}(t_{1})\big[e^{M_{b}(t_{2}-t_{1})}-1\big].
\end{align*}
\end{enumerate}
\end{proof}

The estimation given in the following Lemma is the key step of the proof of
main Theorem.

\begin{lemma}\label{lemEstimationMain}
For any $t_{1}<t_{2}$,
\begin{align}\label{eqnBasicInequality}
|x(t_{2})|&\le
M_{0}(t_{1})e^{M_{b}\mu_{-}(t_{1},t_{2})}-\big[M_{0}(t_{1})-\max\{\delta
M_{0}(t_{1}),|x(t_{1})|
-\frac{\eta}{2}\mu_{\eta}(t_{1},t_{2})M_{0}(t_{1})\}\big]
e^{-M_{a}\mu_{+}(t_{1},t_{2})}.
\end{align}
\end{lemma}
\begin{proof}
We prove this lemma by induction.

Step 1. We verify the initial case that $(t_{1},t_{2})$ is contained
in only one of $S_{+}(t_{1},t_{2})$, $S_{\eta}(t_{1},t_{2})$ and
$S_{-}(t_{1},t_{2})$. There are totally three cases corresponding to
those considered in Lemma \ref{lemEstimationMain1}.
\begin{enumerate}
  \item
  $(t_{1},t_{2})=S_{+}(t_{1},t_{2})$. In this case,
  \eqref{eqnBasicInequality} reduces to
  \begin{align*}
|x(t_{2})|&\le M_{0}(t_{1})-\big[M_{0}(t_{1})-\max\{\delta
M_{0}(t_{1}),|x(t_{1})| \}\big] e^{-M_{a}\mu_{+}(t_{1},t_{2})}\\
&=
M_{0}(t_{1})\big[(1-e^{-M_{a}\mu_{+}(t_{1},t_{2})}\big]+\max\{\delta
M_{0}(t_{1}),|x(t_{1})| \}e^{-M_{a}\mu_{+}(t_{1},t_{2})}.
\end{align*}
This is obvious from case 1 of Lemma \ref{lemEstimationMain1}.

\item
$(t_{1},t_{2})=S_{\eta}(t_{1},t_{2})$. In this case,
\eqref{eqnBasicInequality} reduces to
\begin{align*}
|x(t_{2})|&\le M_{0}(t_{1})-\big[M_{0}(t_{1})-\max\{\delta
M_{0}(t_{1}),|x(t_{1})|
-\frac{\eta}{2}\mu_{\eta}(t_{1},t_{2})M_{0}(t_{1})\}\big]\\
&=\max\{\delta M_{0}(t_{1}),|x(t_{1})|
-\frac{\eta}{2}\mu_{\eta}(t_{1},t_{2})M_{0}(t_{1})\}.
\end{align*}
This can be obtained by case 2 of Lemma \ref{lemEstimationMain1} by
letting $\widetilde{M}=M_{0}(t_{1})$.

\item
$(t_{1},t_{2})=S_{-}(t_{1},t_{2})$. In this case,
\eqref{eqnBasicInequality} reduces to
\begin{align*}
|x(t_{2})|&\le
M_{0}(t_{1})e^{M_{b}\mu_{-}(t_{1},t_{2})}-\big[M_{0}(t_{1})-\max\{\delta
M_{0}(t_{1}),|x(t_{1})| \}\big]\\
&=\max\{\delta
M_{0}(t_{1}),|x(t_{1})|\}+M_{0}(t_{1})[e^{M_{b}\mu_{-}(t_{1},t_{2})}-1]
\end{align*}
This is also obvious from case 3 of Lemma \ref{lemEstimationMain1}.
\end{enumerate}

Step 2. Assume that \eqref{eqnBasicInequality} holds for a given
$(t_{1}$, $t_{2})$, and consider $(t_{1}$, $t_{3})$, where $t_{3}>t_{2}$
and $(t_{2},t_{3})$ is contained in only one of
$S_{+}(t_{2},t_{3})$, $S_{\eta}(t_{2},t_{3})$ and
$S_{-}(t_{2},t_{3})$. There are also three cases.
\begin{enumerate}
  \item
  Case 1. $(t_{2},t_{3})=S_{+}(t_{2},t_{3})$. In this case, from Lemma
  \ref{lemEstimationMain1}, we have
  \begin{align*}
    &|x(t_{3})|\le
    M_{0}(t_{2})-[M_{0}(t_{2})-|x(t_{2})|]e^{-M_{a}\mu_{+}(t_{2},t_{3})}\\
    &=M_{0}(t_{2})\big[1-e^{-M_{a}\mu_{+}(t_{2},t_{3})}\big]+|x(t_{2})|e^{-M_{a}\mu_{+}(t_{2},t_{3})}\\
    &\le M_{0}(t_{1})e^{M_{b}\mu_{-}(t_{1},t_{2})}\big[1-e^{-M_{a}\mu_{+}(t_{2},t_{3})}\big]
    +\bigg\{M_{0}(t_{1})e^{M_{b}\mu_{-}(t_{1},t_{2})}\\
    &-\bigg[M_{0}(t_{1})-\max\{\delta
M_{0}(t_{1}),|x(t_{1})|
-\frac{\eta}{2}\mu_{\eta}(t_{1},t_{2})M_{0}(t_{1})\}\bigg]
e^{-M_{a}\mu_{+}(t_{1},t_{2})}\bigg\}e^{-M_{a}\mu_{+}(t_{2},t_{3})}\\
&=M_{0}(t_{1})e^{M_{b}\mu_{-}(t_{1},t_{2})}-\big[M_{0}(t_{1})-\max\{\delta
M_{0}(t_{1}),|x(t_{1})|
-\frac{\eta}{2}\mu_{\eta}(t_{1},t_{2})M_{0}(t_{1})\}\big]
e^{-M_{a}\mu_{+}(t_{1},t_{3})}\\
&=M_{0}(t_{1})e^{M_{b}\mu_{-}(t_{1},t_{3})}-\big[M_{0}(t_{1})-\max\{\delta
M_{0}(t_{1}),|x(t_{1})|
-\frac{\eta}{2}\mu_{\eta}(t_{1},t_{3})M_{0}(t_{1})\}\big]
e^{-M_{a}\mu_{+}(t_{1},t_{3})}.
  \end{align*}

  \item
  Case 2. $(t_{2},t_{3})=S_{\eta}(t_{2},t_{3})$. In this case, let
  $\widetilde{M}=\max\{M_{0}(t_{1}),M_{0}(t_{2})\}$,
  then $M_{0}(t_{1})\le\widetilde{M}\le
  M_{0}(t_{1})e^{M_{b}\mu_{-}(t_{1},t_{2})}=M_{0}(t_{1})e^{M_{b}\mu_{-}(t_{1},t_{3})}$.

  From Lemma \ref{lemEstimationMain1}, we have
  \begin{align*}
    |x(t_{3})|\le &\max\{\delta
    \widetilde{M},|x(t_{2})|-\frac{\eta}{2}\mu_{\eta}(t_{2},t_{3})\widetilde{M}\}\\
    \le&\max\{\delta
    M_{0}(t_{1})e^{M_{b}\mu_{-}(t_{1},t_{3})},|x(t_{2})|-\frac{\eta}{2}\mu_{\eta}(t_{2},t_{3})M_{0}(t_{1})\}\\
    \le &\max\{\delta
    M_{0}(t_{1})e^{M_{b}\mu_{-}(t_{1},t_{3})},M_{0}(t_{1})e^{M_{b}\mu_{-}(t_{1},t_{2})}
    -\big[M_{0}(t_{1})-\max\{\delta
M_{0}(t_{1}),\\
&|x(t_{1})| -\frac{\eta}{2}\mu_{\eta}(t_{1},t_{2})M_{0}(t_{1})\}\big]
e^{-M_{a}\mu_{+}(t_{1},t_{2})}-\frac{\eta}{2}\mu_{\eta}(t_{2},t_{3})M_{0}(t_{1})\}\\
\le &\max\{\delta
    M_{0}(t_{1})e^{M_{b}\mu_{-}(t_{1},t_{3})},M_{0}(t_{1})e^{M_{b}\mu_{-}(t_{1},t_{3})}
    -\big[M_{0}(t_{1})-\max\{\delta
M_{0}(t_{1}),\\
&|x(t_{1})| -\frac{\eta}{2}\mu_{\eta}(t_{1},t_{3})M_{0}(t_{1})\}\big]
e^{-M_{a}\mu_{+}(t_{1},t_{3})}\big\}\\
=& M_{0}(t_{1})e^{M_{b}\mu_{-}(t_{1},t_{3})}-\min\{(1-\delta)
    M_{0}(t_{1})e^{M_{b}\mu_{-}(t_{1},t_{3})},\big[M_{0}(t_{1})-\max\{\delta
M_{0}(t_{1}),\\
&|x(t_{1})| -\frac{\eta}{2}\mu_{\eta}(t_{1},t_{3})M_{0}(t_{1})\}\big]
e^{-M_{a}\mu_{+}(t_{1},t_{3})}\}\\
\le & M_{0}(t_{1})e^{M_{b}\mu_{-}(t_{1},t_{3})}-\min\{(1-\delta)
    M_{0}(t_{1}),\big[M_{0}(t_{1})-\max\{\delta
M_{0}(t_{1}),\\
&|x(t_{1})| -\frac{\eta}{2}\mu_{\eta}(t_{1},t_{3})M_{0}(t_{1})\}\big]
e^{-M_{a}\mu_{+}(t_{1},t_{3})}\}\\
\le & M_{0}(t_{1})e^{M_{b}\mu_{-}(t_{1},t_{3})}-\min\{(1-\delta)
    M_{0}(t_{1}),M_{0}(t_{1})-\max\{\delta
M_{0}(t_{1}),\\
&|x(t_{1})| -\frac{\eta}{2}\mu_{\eta}(t_{1},t_{3})M_{0}(t_{1})\}
\}e^{-M_{a}\mu_{+}(t_{1},t_{3})}\\
=& M_{0}(t_{1})e^{M_{b}\mu_{-}(t_{1},t_{3})}-\big[
    M_{0}(t_{1})-\max\{\delta M_{0}(t_{1}),\max\{\delta
M_{0}(t_{1}),\\
&|x(t_{1})| -\frac{\eta}{2}\mu_{\eta}(t_{1},t_{3})M_{0}(t_{1})\}\}
\big]e^{-M_{a}\mu_{+}(t_{1},t_{3})}\\
=& M_{0}(t_{1})e^{M_{b}\mu_{-}(t_{1},t_{3})}-\big[
    M_{0}(t_{1})-\max\{\delta M_{0}(t_{1}),
|x(t_{1})|\\
& -\frac{\eta}{2}\mu_{\eta}(t_{1},t_{3})M_{0}(t_{1})\}
\big]e^{-M_{a}\mu_{+}(t_{1},t_{3})}
  \end{align*}

  \item
  Case 3. $(t_{2},t_{3})=S_{-}(t_{2},t_{3})$. In this case, from Lemma
  \ref{lemEstimationMain1}, we have
  \begin{align*}
  |x(t_{3})|\le&
  |x(t_{2})|+M_{0}(t_{2})\big[e^{M_{b}\mu_{-}(t_{2},t_{3})}-1\big]\\
  \le&M_{0}(t_{1})e^{M_{b}\mu_{-}(t_{1},t_{2})}-\big[M_{0}(t_{1})-\max\{\delta
M_{0}(t_{1}),|x(t_{1})|\\
&-\frac{\eta}{2}\mu_{\eta}(t_{1},t_{2})M_{0}(t_{1})\}\big]
e^{-M_{a}\mu_{+}(t_{1},t_{2})}
+M_{0}(t_{1})e^{M_{b}\mu_{-}(t_{1},t_{2})}\big[e^{M_{b}\mu_{-}(t_{2},t_{3})}-1\big]\\
=&M_{0}(t_{1})e^{M_{b}\mu_{-}(t_{1},t_{3})}-\big[M_{0}(t_{1})-\max\{\delta
M_{0}(t_{1}),|x(t_{1})|\\
&-\frac{\eta}{2}\mu_{\eta}(t_{1},t_{3})M_{0}(t_{1})\}\big]
e^{-M_{a}\mu_{+}(t_{1},t_{3})}
  \end{align*}
  The proof is completed.
\end{enumerate}

\end{proof}

Based on the estimation (\ref{eqnBasicInequality}) given in Lemma
\ref{lemEstimationMain}, we are to prove Theorem \ref{thmMain}.

{\bf\it Proof of Theorem 1.}

For $t\in [t_{1}^{-},t_{1}]$, from Lemma \ref{lemEstimationMain},
and noting $|x(t_{0})|\le M_{0}(t_{0})$, we have
\begin{align*}
|x(t)|&\le M_{0}(t_{0})e^{M_{b}\mu_{-}(t_{0},t)}-\Big[M_{0}(t_{0})
-\max\Big\{\delta
M_{0}(t_{0}),|x(t_{0})|-\frac{\eta}{2}\mu_{\eta}(t_{0},t)M_{0}(t_{0})\Big\}\Big]
e^{-M_{a}\mu_{+}(t_{0},t)}\\
&\le M_{0}(t_{0})e^{M_{b}\mu_{-}(t_{0},t)}-\Big[M_{0}(t_{0})
-\max\Big\{\delta
M_{0}(t_{0}),M_{0}(t_{0})-\frac{\eta}{2}\mu_{\eta}(t_{0},t)M_{0}(t_{0})\Big\}\Big]
e^{-M_{a}\mu_{+}(t_{0},t)}\\
&=M_{0}(t_{0})\bigg[e^{M_{b}\mu_{-}(t_{0},t)}
-\min\{1-\delta,\frac{\eta}{2}\mu_{\eta}(t_{0},t)\}e^{-M_{a}\mu_{+}(t_{0},t)}\bigg]\\
&=M_{0}(t_{0})\bigg[e^{M_{b}\mu_{-}(t_{0},t)}
-\frac{\eta}{2}\min\{\frac{1}{M_{a}},\mu_{\eta}(t_{0},t)\}e^{-M_{a}\mu_{+}(t_{0},t)}\bigg]
\end{align*}
This implies
\begin{align*}
M_{0}(t_{1})\le&
M_{0}(t_{0})\Big[e^{M_{b}\mu_{-}(t_{0},t_{1})}-\frac{\eta}{2}\min\Big\{\frac{1}{M_{a}},
\mu_{\eta}(t_{0},t_{1}^{-})\Big\}e^{-M_{a}\mu_{+}(t_{0},t_{1})}\Big]\\
\le&M_{0}(t_{0})\Big[e^{M_{b}\mu_{-}(t_{0},t_{1})}-\frac{\eta}{2}\min\Big\{\frac{1}{M_{a}},
\mu_{\eta}(t_{0},t_{1}^{-})\Big\}e^{-M_{a}(N+1)\tau_{\max}}\Big]
\end{align*}
Repeating this process, we have
\begin{align*}
&M_{0}(t_{m})\le
M_{0}(t_{0})\prod_{k=0}^{m-1}\Big[e^{M_{b}\mu_{-}(t_{k},t_{k+1})}
-\frac{\eta}{2}\min\Big\{\frac{1}{M_{a}},\mu_{\eta}(t_{k},
t_{k+1}^{-})\Big\} e^{-M_{a}(N+1)\tau_{\max}} \Big]
\end{align*}
Under the condition \eqref{eqnConditionMain}, for a given $C\in
(C^{*},\eta/2)$, we can choose $k^{*}$ large enough such that for
all $k\ge k^{*}$,
\begin{align*}
\frac{[e^{M_{b}\mu_{-}(t_{k},t_{k+1})}-1] e^{M_{a}(N+1)\tau_{\max}}}
{\min\{\frac{1}{M_{a}},\mu_{\eta}(t_{k},t_{k+1}^{-})\}}\le C
\end{align*}
which implies
\begin{align*}
e^{M_{b}\mu_{-}(t_{k},t_{k+1})}\le
1+C\min\{\frac{1}{M_{a}},\mu_{\eta}(t_{k},t_{k+1}^{-})\}e^{-M_{a}(N+1)\tau_{\max}}
\end{align*}
Thus for $m>k^{*}$,
\begin{align}\label{eqnKstar}
&M_{0}(t_{m})\le
M_{0}(t_{k^{*}})\prod_{k=k^{*}-1}^{m-1}\Big[1-(\frac{\eta}{2}-C)
\min\{\frac{1}{M_{a}},\mu_{\eta}(t_{k},t_{k+1}^{-})\}e^{-M_{a}(N+1)\tau_{\max}}\Big]
\end{align}

By the condition \eqref{eqnConditionMaina}, we have
\begin{align}\label{eqnConditionInfinity}
\sum_{k=0}^{+\infty}\min\{\frac{1}{M_{a}},\mu_{\eta}(t_{k},
t_{k+1}^{-})\}e^{-M_{a}(N+1)\tau_{\max}}=\infty.
\end{align}

Thus
\begin{eqnarray*}
\lim_{m\to+\infty}M_{0}(t_{m})=0
\end{eqnarray*}

For $t\in [t_{m}^{-},t_{m}]$, $|x(t)|\le M_{0}(t_{m})$, and for
$t\in [t_{m},t_{m+1}^{-}]$, we have the estimation that
\begin{eqnarray}\label{eqnXestimation}
|x(t)|\le M_{0}(t_{m})e^{M_{b}N\tau_{\max}}.
\end{eqnarray}
Therefore,we have
\begin{eqnarray*}
\lim_{t\to+\infty}|x(t)|=0.
\end{eqnarray*}
On the other hand, when $a(\cdot)$, $b(\cdot)$ are given, we can
find a fixed $k^{*}$ for \eqref{eqnKstar}. Thus let $K'=\max_{1\le
k\le k^{*}}M_{0}(t_{k})/M_{0}(t_{0})$, we have
$$M_{0}(t_{m})\le K'M_{0}(0)=\max_{-\tau_{\max}\le s\le 0}K'|\phi(s)|$$
for each $m$. Let $K=K'e^{M_{b}N\tau_{\max}}$, then
$$|x(t)|\le K\max_{-\tau_{\max}\le s\le 0}|\phi(s)|$$ for each $t$.

Furthermore, if there exists $\epsilon>0$ such that
$\mu_{\eta}(t_{k},t_{k+1})\ge \epsilon$, then
$$
1-(\frac{\eta}{2}-C)
\min\{\frac{1}{M_{a}},\mu_{\eta}(t_{k},t_{k+1}^{-})\}e^{-M_{a}(N+1)\tau_{\max}}
\le
1-(\frac{\eta}{2}-C)\min\{\frac{1}{M_{a}},\epsilon\}e^{-M_{a}(N+1)\tau_{\max}}\triangleq
\lambda_{0}<1.
$$
Then from \eqref{eqnKstar} we have for $m\ge k^{*}$,
\begin{eqnarray*}
M_{0}(t_{m})\le M_{0}(t_{k^{*}})\lambda_{0}^{m-k^{*}}.
\end{eqnarray*}
Thus for $t\in [t_{m},t_{m+1}]$, we have the estimation
\begin{eqnarray*}
|x(t)|&\le& M_{0}(t_{m})e^{M_{b}(N+1)\tau_{\max}}\le
M_{0}(t_{k^{*}})e^{M_{b}(N+1)\tau_{\max}}\lambda_{0}^{m-k^{*}} \le
K'M_{0}(t_{0})e^{M_{b}(N+1)\tau_{\max}}\lambda_{0}^{-(k^{*}+1)}e^{(m+1)\ln
\lambda_{0}}\\
&=&K'M_{0}(t_{0})e^{M_{b}(N+1)\tau_{\max}}\lambda_{0}^{-(k^{*}+1)}e^{\frac{\ln
\lambda_{0}}{(N+1)\tau_{\max}}(m+1)(N+1)\tau_{\max}}\\
&\le&K'M_{0}(t_{0})e^{M_{b}(N+1)\tau_{\max}}\lambda_{0}^{-(k^{*}+1)}e^{\frac{\ln
\lambda_{0}}{(N+1)\tau_{\max}}t},
\end{eqnarray*}
where the last inequality comes from the fact that $\frac{\ln
\lambda_{0}}{(N+1)\tau_{\max}}<0$ and $t\le (m+1)(N+1)\tau_{\max}$.
Let $\alpha=-\frac{\ln\lambda_{0}}{(N+1)\tau_{\max}}>0$, and
$\widetilde{K}=\max\big\{\max_{t_{0}\le t\le
t_{k^{*}}}\frac{|x(t)|}{M_{0}(t)},K'e^{M_{b}(N+1)\tau_{\max}}\lambda_{0}^{-(k^{*}+1)}\big\}e^{\alpha
k^{*}(N+1)\tau_{\max}}$, then
\begin{eqnarray*}
  |x(t)|\le \widetilde{K}\max_{-\tau_{\max}\le s\le 0}|\phi(s)| e^{-\alpha
  t}.
\end{eqnarray*}

The proof is completed.

\section{Applications }\label{secApplication} In this section, we will give two
applications the theoretical results, including self synchronization
in a class of neural networks with time varying delays, and the
existence and exponential stability of periodic solutions of a class
of neural networks with periodic coefficients and delays.
\subsection{Self synchronization of neural
networks}\label{secSelfSyn} First, we apply the theoretical results
obtained in previous sections to the self synchronization analysis
of neural networks. In Liu, Lu \& Chen(2011), we have discussed
almost sure self synchronization in neural networks with randomly
switching connections without time delays. In this paper, we discuss
self synchronization in neural networks with bounded time-varying
delays.

To be more general, consider the following Volterra functional
differential systems
\begin{align}\label{eqnOuterSyn}
\frac{dx_{i}(t)}{dt}=&-d_{i}(t)x_{i}(t)
+f_{i}(x_{1},\cdots,x_{n},x_{1}(t-\tau_{i1}(t)),\cdots,
x_{n}(t-\tau_{in}(t)),t)+I_{i}(t),~~i=1,\cdots,n,
\end{align}
where
\begin{enumerate}
\item[(i)] 
\begin{eqnarray*}
\bigg|\frac{\partial f_{i}(u_{1},\cdots,u_{n},v_{1},\cdots,v_{n},t)}{\partial u_{j}}\bigg|\le A_{ij}(t),\\
\bigg|\frac{\partial f_{i}(u_{1},\cdots,u_{n},v_{1},\cdots,v_{n},t)}{\partial v_{j}}\bigg|\le B_{ij}(t);
\end{eqnarray*}
\item[(ii)]
$\tau_{ij}(t)\le \tau_{\max},~~i,j=1,2,\cdots,n$.
\end{enumerate}
Before state our main result, we first extend the $\eta$-condition
to the case of $n$ function pairs.
\begin{definition}[Common $\eta$-condition] Given $n$ function
pairs $\{a_{i}(\cdot),b_{i}(\cdot)\}_{i=1}^{n}$ with
$0<a_{i}(\cdot)\le M_{a}$, $|b_{i}(\cdot)|\le M_{b}$, we say they
satisfy the {\bf common $\eta$-condition} if the function pair
$\{a(\cdot),b(\cdot)\}$ satisfies the $\eta$-condition, where
$a(t)=a_{i_{t}}(t), b(t)=b_{i_{t}}(t)$, with $i_{t}$ satisfying
$a_{i_{t}}(t)-|b_{i_{t}}(t)|=\min_{j}\{a_{j}(t)-|b_{j}(t)|\}$ for
each $t$.
\end{definition}

Then we have
\begin{theorem}\label{thmMain1}
Suppose that $0<d_{i}(t)-\sum_{j=1}^{n}A_{ij}(t)\le M_{a}$,
$\sum_{j=1}^{n}B_{ij}(t)\le M_{b}$ for any $t\ge 0$. If
$\{d_{i}(\cdot)-\sum_{j=1}^{n}A_{ij}(\cdot),\sum_{j=1}^{n}B_{ij}(\cdot)\}_{i=1}^{n}$
satisfy the common $\eta$-condition for a constant $\eta>0$, then
the network \eqref{eqnOuterSyn} will reach outer self
synchronization, i.e., for any two initial values $\phi(s),
\psi(s)\in \mathbb{R}^{n}$, $s\in[-\tau_{max},0]$, the trajectories
with initial values $\phi(s)$, $\psi(s)$ respectively will satisfy
\begin{eqnarray*}
\lim_{t\to\infty}\|x(t)-y(t)\|=0.
\end{eqnarray*}
\end{theorem}

\begin{proof}
Let $\tau(t)=\max_{i,j=1,\cdots,n}\tau_{ij}(t)$, $z_{i}(t)=|x_{i}(t)-y_{i}(t)|$,
and $z(t)=\max_{i}\{z_{i}(t)\}$, then we have
  \begin{eqnarray*}
&&D^{+}z_{i}(t)=-\sign(x_{i}(t)-y_{i}(t))d_{i}(t)[x_{i}(t)-y_{i}(t)]
+\sign(x_{i}(t)-y_{i}(t))[f_{i}(x_{1}(t),\cdots,x_{n}(t),\\
&&x_{1}(t-\tau_{i1}(t)),\cdots,x_{n}(t-\tau_{in}(t)),t)-f_{i}(y_{1}(t),\cdots,y_{n}(t),y_{1}(t-\tau_{i1}(t)),
\cdots,y_{n}(t-\tau_{in}(t)),t)]\\
&&\le
-d_{i}(t)z_{i}(t)+\sum_{j=1}^{n}A_{ij}(t)z_{j}(t)+\sum_{j=1}^{n}B_{ij}(t)z_{j}(t-\tau_{ij}(t))
\end{eqnarray*}
which implies
\begin{align*}
  D^{+}z(t)\le -[d_{i_{t}}(t)-\sum_{j=1}^{n}A_{i_{t}j}(t)]z(t)
  +\sum_{j=1}^{n}B_{i_{t}j}(t)\sup_{t-\tau(t)\le s \le t}z(s)
\end{align*}
It is easy to see that
$\{d_{i_{t}}(\cdot)-\sum_{j=1}^{n}A_{i_{t}j}(\cdot),\sum_{j=1}^{n}B_{i_{t}j}\}$
satisfies the $\eta$-condition with the same $\eta$ if
$\{d_{i}(\cdot)-\sum_{j=1}^{n}A_{ij}(\cdot),\sum_{j=1}^{n}B_{ij}(\cdot)\}$
satisfies the $\eta$-condition, by Theorem \ref{thmMain}, we
conclude
\begin{align}
\lim_{t\rightarrow\infty}z(t)=0,
\end{align}
which implies
\begin{eqnarray*}
    \lim_{t\to\infty}\|x(t)-y(t)\|=0.
  \end{eqnarray*}
\end{proof}

In particular, if $I_{i}(t)\equiv 0$ and $f_{i}(0,\cdots,0)=0$ for
$i=1,\cdots,n$. Then $x=0$ is a equilibrium. As a direct consequence
of Theorem 2, we have
\begin{corollary}
  Under the conditions in Theorem \ref{thmMain1}, if $I_{i}(t)\equiv 0$,
  and $f_{i}(0,\cdots,0)=0$ for each $i$, then the equilibrium $0$ of
  \eqref{eqnOuterSyn} is globally asymptotically stable.
\end{corollary}

\subsection{Periodic neural networks}\label{secPeriodicNN}
As another application, we discuss periodic neural networks, which
can be described as
\begin{eqnarray}
\frac{du_i(t)}{dt}=-d_{i}(t)u_{i}(t)+\sum_{j=1}^{n}a_{ij}(t)g_j(u_j(t))
+\sum_{j=1}^{n}b_{ij}(t)f_{j}(u_{j}(t-\tau_{ij}(t)))+I_i(t)\quad
 \label{periodic}
\end{eqnarray}
where $d_{i}(t)>d_{i}>0$, $a_{ij}(t), b_{ij}(t), \tau_{ij}(t)>0,
I_i(t): \mathbb{R}^{+}\rightarrow \mathbb{R}$ are continuously
periodic functions with period $\omega>0$, $i,j=1,2,\ldots,n$.

By using a maximum function and the Brouwer fixed point theorem, it
was proved in Lu \& Chen(2004) that
\begin{proposition}
Under the conditions that for $i=1,\cdots,n$,
$|g_{i}(x+h)-g_{i}(x)|\le G_{i}|h|$, $|f_{i}(x+h)-f_{i}(x)|\le
F_{i}|h|$ and $-d_{i}(t)+\sum\limits_{j=1
}^{n}G_{j}|a_{ij}(t)|+\sum\limits_{j=1}^{n}F_{j}|b_{ij}(t)| <-\eta$,
the system (\ref{periodic}) has an $\omega-$periodic solution
$x(t)$, and there exists $\alpha>0$ such that for any solution
$u(t)=[u_{1}(t),\cdots,u_{n}(t)]$, we have
\begin{eqnarray}
||u(t)-x(t)||=O(e^{-\alpha t }),\quad t\rightarrow\infty.
\end{eqnarray}

\end{proposition}

Obviously, if the requirement $-d_{i}(t)+\sum\limits_{j=1
}^{n}G_{j}|a_{ij}(t)|+\sum\limits_{j=1}^{n}F_{j}|b_{ij}(t)| <-\eta$,
$i=1,2,\cdots,n$  is not satisfied, then the Brouwer fixed point
theorem is no longer applicable. Here, as application of the
theoretical results, we can prove the same result without this
requirement. First, we make the following assumption.
\begin{assumption}\label{assumPeriodic}
\begin{enumerate}
\item $d_{i}(t)>d_{i}>0$, $a_{ij}(t), b_{ij}(t),
0<\tau_{ij}(t)\le \tau_{\max}, I_i(t)$ are continuously periodic
functions of $t$ with period $\omega>0$, $i,j=1,2,\ldots,n$;
\item There exist $G_{i}>0$, $F_{i}>0$, $i=1,\cdots,n$, such that
$|g_{i}(x+h)-g_{i}(x)|\le G_{i}|h|$, $|f_{i}(x+h)-f_{i}(x)|\le
F_{i}|h|$ for each $h\in \mathbb{R}$;
\item There exists $M_{a}>0$, $M_{b}>0$ such that
$0\le d_{i}(t)-\sum\limits_{j=1 }^{n}G_{j}|a_{ij}(t)|\le M_{a},
\sum\limits_{j=1}^{n}F_{j}|b_{ij}(t)|\le M_{b}$;
\end{enumerate}
\end{assumption}
For some $\eta>0$, denote $\overline{S}_{\eta}=\{t\in
[0,\omega]:~d_{i}(t)-\sum_{j=1}^{n}\big[G_{j}|a_{ij}(t)|+F_{j}|b_{ij}(t)|\big]\ge
\eta,~~i=1,\cdots,n\}$, and
$\overline{\mu}_{\eta}=\mu(\overline{S}_{\eta})$. Denote
$\overline{S}_{-}=\{t\in
[0,\omega]:~d_{i}(t)-\sum_{j=1}^{n}\big[G_{j}|a_{ij}(t)|+F_{j}|b_{ij}(t)|\big]<0
\text{ for some }i\}$, and
$\overline{\mu}_{-}=\mu(\overline{S}_{-})$. Let $p$ be the smallest
integer such that $\tau_{\max}\le p\omega$.
\begin{theorem}\label{thmPeriodic}
Under Assumption \ref{assumPeriodic}, if there exists some $\eta>0$,
$N>0$ such that $\mu(\overline{S}_{\eta})>0$ and
\begin{eqnarray}\label{conditionPeriodic}
\frac{\big[e^{M_{b}p(N+1)\overline{\mu}_{-}}-1\big]e^{M_{a}p(N+1)\omega}}{\min\{\frac{1}{M_{a}},pN\overline{\mu}_{\eta}\}}<\frac{\eta}{2},
\end{eqnarray}
then the periodic neural network (\ref{periodic}) has an
$\omega-$periodic solution $x(t)$. Furthermore, for any solution
$u(t)=[u_{1}(t),\cdots,u_{n}(t)]$, we have
\begin{eqnarray}
||u(t)-x(t)||=O(e^{-\alpha t }),\quad t\rightarrow\infty.
\end{eqnarray}

\end{theorem}

\begin{proof}
First, we prove the existence of an $\omega$ period solution $x(t)$.
Let $u(t)=[u_{1}(t),\cdots,u_{n}(t)]^{\top}$ be an arbitrary
solution of \eqref{periodic}. Let $\bar{u}(t)=u(t)-u(t-\omega)$,
then we have:
\begin{eqnarray*}
\frac{d\bar{u}_{i}(t)}{dt}=-d_{i}(t)\bar{u}_{i}(t)
+\sum_{j=1}^{n}a_{ij}(t)[g_{j}(u_{j}(t))-g_{j}(u_{j}(t-\omega))]\\
+\sum_{j=1}^{n}b_{ij}(t)[f_{j}(u_{j}(t-\tau_{ij}(t)))-f_{j}(u_{j}(t-\tau_{ij}(t)-\omega))].
\end{eqnarray*}
Let $v(t)=\max_{i}\{|\bar{u}_{i}(t)|\}$, and denote $i_{t}$ the
index such that $v(t)=|\bar{u}_{i_{t}}(t)|$, then we have:
\begin{eqnarray*}
D^{+}v(t)\le
-d_{i_{t}}(t)v(t)+\sum_{j=1}^{n}|G_{j}a_{i_{t}j}(t)|v(t)+\sum_{j=1}^{n}F_{j}|b_{i_{t}j}(t)|\sup_{0\le
s\le p\omega}v(t-s).
\end{eqnarray*}
From definition, we have $S_{\eta}(0,\omega)\supseteq
\overline{S}_{\eta}$ and $S_{-}(0,\omega)\subseteq
\overline{S}_{-}$. Let $\tau_{\max}=p\omega$, from Theorem
\ref{thmMain}, $v(t)$ will converge to zero exponentially. Then for
any given $t^{*}$, the sequence
$\{\sum_{m=1}^{k}\bar{u}(t^{*}+k\omega)\}$, $k=1,2,3,\cdots$ is a
Cauchy sequence. Thus there exists $x(t^{*})\in \mathbb{R}^{n}$ such
that $\lim_{k\to
\infty}\sum_{m=1}^{k}\bar{u}(t^{*}+k\omega)=x(t^{*})$. From the
definition of $\bar{u}(t)$, this implies
\begin{eqnarray*}
\lim_{k\to\infty}u(t^{*}+k\omega)=x(t^{*}).
\end{eqnarray*}
Since
$\lim_{k\to\infty}u(t^{*}+k\omega)=\lim_{k\to\infty}u(t^{*}+\omega+k\omega)$,
it is easy to see that $x(t^{*})=x(t^{*}+\omega)$.

For any $t_{1}>0$, $t_{2}>0$,
\begin{eqnarray*}
x_{i}(t_{2})-x_{i}(t_{1})&=&\lim_{k\to\infty}[u_{i}(t_{2}+k\omega)-u_{i}(t_{1}+k\omega)]\\
&=&\lim_{k\to\infty}\int_{t_{1}}^{t_{2}}\big[-d_{i}(t)u_{i}(t+k\omega)
+\sum_{j=1}^{n}a_{ij}(t)g_{j}(u_{j}(t+k\omega))\\
&&+\sum_{j=1}^{n}b_{ij}(t)f_{j}(u_{j}(t+k\omega-\tau_{ij}(t)))\big]dt\\
&=&\int_{t_{1}}^{t_{2}}\big[-d_{i}(t)x_{i}(t)+\sum_{j=1}^{n}a_{ij}(t)g_{j}(x_{j}(t))\\
&&+\sum_{j=1}^{n}b_{ij}(t)f_{j}(x_{j}(t-\tau_{ij}(t)))\big]dt.
\end{eqnarray*}
Therefore, $x_{i}(t)$ is absolutely continuous and
\begin{eqnarray*}
\frac{dx_{i}(t)}{dt}&=&-d_{i}(t)x_{i}(t)+\sum_{j=1}^{n}a_{ij}(t)g_{j}(x_{j}(t))
+\sum_{j=1}^{n}b_{ij}(t)f_{j}(x_{j}(t-\tau_{ij}(t))).
\end{eqnarray*}
Thus, $x(t)$ is a periodic solution of \eqref{periodic}.

Again, let $u(t)=[u_{1}(t),\cdots,u_{n}(t)]^{\top}$ be an arbitrary
solution of \eqref{periodic}. Denote $\tilde{u}(t)=u(t)-x(t)$, and
$\tilde{v}(t)=\max_{i}\{|\tilde{u}_{i}(t)|\}$. Then, using an
argument similar as above, we can show that $\tilde{v}(t)$ tends to
zero exponentially. This implies \begin{eqnarray}
||u(t)-x(t)||=O(e^{-\alpha t }),\quad t\rightarrow\infty
\end{eqnarray}
for some $\alpha>0$. The proof is completed.
\end{proof}

From Theorem \ref{thmPeriodic}, we can have the following corollary.
\begin{corollary}
Under Assumption \ref{assumPeriodic}, if
$d_{i}(t)-\sum_{j=1}^{n}|a_{ij}(t)|G_{j}-\sum_{j=1}^{n}|b_{ij}(t)|F_{j}\ge
0$, and there exists $\eta>0$ such that
$\mu(\overline{S}_{\eta})>0$, then the periodic neural network
\eqref{periodic} has an $\omega$-period solution which is
exponentially asymptotically stable.
\end{corollary}

\section{Numerical Examples}\label{secExample}
In this section, we provide two simple examples with simulation to
illustrate the theoretical results.

\subsection{Delay differential system}

We consider the following delay differential system:
\begin{eqnarray}\label{eqnExample}
\dot{x}(t)=-a(t)x(t)+b(t)x(t-\tau(t)).
\end{eqnarray}
Here we take $\tau(t)=t-\lfloor t\rfloor$, where $\lfloor t\rfloor$
denotes the largest integer that is no greater than $t$. Let
$a(t)\equiv 1$, and $b(t)$ be a step function such that
\begin{eqnarray*}
  b(t)=\left\{
  \begin{array}{rl}
    0.8, & t\in [2k,2k+0.5],\\
    1.2, & t\in [2k+1, 2k+1.002]\\
    1, & \text{otherwise}.
  \end{array}
  \right.
\end{eqnarray*}
Thus, $\tau_{\max}=1$, $M_{a}=1$, and $M_{b}=1.2$. We take
$\eta=0.2$, then
$S_{\eta}(0,+\infty)=\cup_{k=0}^{+\infty}[2k,2k+0.5]$, and
$S_{-}(0,+\infty)=\cup_{k=0}^{+\infty}[2k+1,2k+1+0.002]$. Let
$t_{0}=0$, and $N=1$, then $\mu_{-}(t_{k},t_{k+1})=0.004$,
$\mu_{\eta}(t_{k},t_{k+1}^{-})=0.5$, and we have
\begin{eqnarray*}
\frac{[e^{M_{b}\mu_{-}(t_{k},t_{k+1})}-1]
e^{M_{a}\mu_{+}(N+1)\tau_{\max}}}
{\min\{\frac{1}{M_{a}},\mu_{\eta}(t_{k},t_{k+1}^{-})\}}
=\frac{(e^{1.2\times 0.004}-1)e^{2}}{\min\{1,0.5\}}\simeq 0.0711
<0.1= \frac{\eta}{2}.
\end{eqnarray*}

Then from Theorem \ref{thmMain}, the zero solution of
\eqref{eqnExample} is asymptotically stable. The simulation results
are provided in Fig. \ref{figSimu1}, where the initial value are
chosen randomly.
\begin{figure}
\centering
\includegraphics[width=0.5\textwidth]{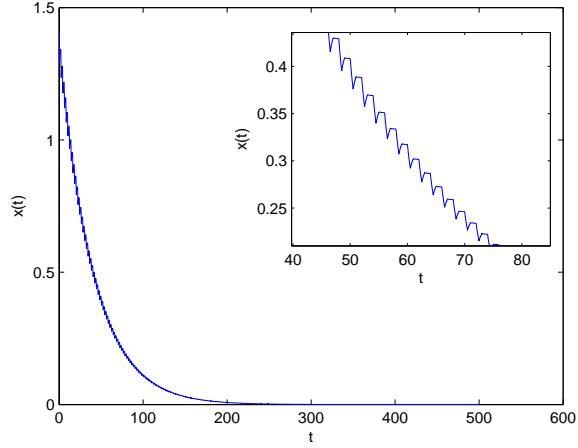}
\caption{Asymptotic stability of the zero solution of Eq.
\eqref{eqnExample}}. \label{figSimu1}
\end{figure}

\subsection{Periodic neural
network with delays} In this simulation, we consider the following
delay periodic neural network with $3$ neurons:
\begin{eqnarray*}
&&\frac{dx_{i}(t)}{dt}=-(2+\sin^{2}(\pi t))x_{i}(t)
+|\sin^{3}(\pi t)|\tanh(x_{i}(t))+\sin^{2}(2\pi t)\tanh(x_{i+1}(t))\\
&&+\cos^{2}(2\pi t)\tanh(x_{i+2}(t))+\sin^{2}(4\pi t)\arctan(x_{i+1}(t-|\sin(2\pi t)|))\\
&&+\cos^{2}(4\pi t)\arctan(x_{i+2}(t-|\cos(2\pi t)|))+\sin(i\pi
t),\quad i=1,2,3.
\end{eqnarray*}
Here, $i+1$ and $i+2$ are understood as $i+1 \bmod 3, i+2 \bmod 3$
if they exceed $3$. Now, we verify that the conditions in Theorem
\ref{thmPeriodic} can be satisfied. In accordance to model
\eqref{periodic}, $d_{i}(t)=2+\sin^{2}(\pi t)$,
\begin{eqnarray*}
[a_{ij}(t)]=\left[
  \begin{array}{ccc}
    |\sin^{3}(\pi t)| & \sin^{2}(2\pi t) & \cos^{2}(2\pi t)\\
    \cos^{2}(2\pi t) & |\sin^{3}(\pi t)| & \sin^{2}(2\pi t) \\
    \sin^{2}(2\pi t) & \cos^{2}(2\pi t) & |\sin^{3}(\pi t)|
  \end{array}\right],
  [b_{ij}(t)]=\left[
  \begin{array}{ccc}
    0 & \sin^{2}(4\pi t) & \cos^{2}(4\pi t)\\
    \cos^{2}(4\pi t) & 0 & \sin^{2}(4\pi t) \\
    \sin^{2}(4\pi t) & \cos^{2}(4\pi t) & 0
  \end{array}\right].
\end{eqnarray*} And we can choose $\tau_{\max}=1$, $F_{i}=G_{i}=1$.
Thus,
\begin{eqnarray*}
&0< d_{i}(t)-\sum_{j=1}^{3}G_{j}|a_{ij}(t)|=1+\sin^{2}(\pi
t)-|\sin^{3}(\pi t)|\le \frac{29}{27},\\
&\sum_{j=1}^{3}F_{j}|b_{ij}(t)|=\sum_{j=1}^{3}|b_{ij}(t)|=1.
\end{eqnarray*}
This implies that we can set $M_{a}=29/27$, $M_{b}=1$. On the other
hand,
$d_{i}(t)-\sum_{j=1}^{3}[G_{j}|a_{ij}(t)|+F_{j}|b_{ij}(t)|]=\sin^{2}(\pi
t)-|\sin^{3}(\pi t)|\ge 0$. This means that $\overline{\mu}_{-}=0$,
so the left-hand term in Ineq. \eqref{conditionPeriodic} is $0$ and
Ineq. \eqref{conditionPeriodic} holds for any $\eta>0$. Since the
maximum of
$d_{i}(t)-\sum_{j=1}^{3}[G_{j}|a_{ij}(t)|+F_{j}|b_{ij}(t)|]$ is
$2/27$, we can choose $\eta=1/27$, and from the continuity of
$d_{i}(t)-\sum_{j=1}^{3}[G_{j}|a_{ij}(t)|+F_{j}|b_{ij}(t)|]$, we
have $\mu(\overline{S}_{\eta})>0$. So the requirements in Theorem
\ref{thmPeriodic} are satisfied and this network has a periodic
solution which is asymptotically stable. This is verified by the
simulation results in Fig. \ref{figPeriodicNN}.

\begin{figure}
\centering
\includegraphics[width=0.5\textwidth]{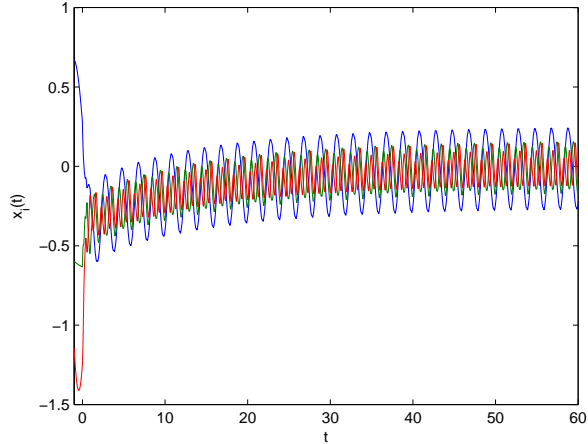}
\caption{Asymptotic stability of periodic solutions in a delayed
periodic neural networks.}\label{figPeriodicNN}
\end{figure}

\section{Conclusions}\label{secConclusion}
In this paper, we discuss generalized Halanay inequality and its
applications. First, we prove a new generalized Halanay inequalities
under less restricted conditions, which are useful for the
asymptotic stability of the zeros solution of a delayed differential
equation. To our knowledge, these conditions are the least
restricted ones known. We also give two applications of the
theoretical results. First, we provide more general sufficient
conditions for the self synchronization of the neural networks with
time varying delays. Then, under more relaxed requirements, we prove
a sufficient condition for the existence and exponential stability
periodical solutions for a class of neural networks with periodic
coefficients and time varying delays. Yet, we only consider bounded
time varying delays. The case of unbounded time-varying delays is
also very important and will be our next research topic.

\noindent{\bf\Large References}
\begin{description}

\item  
Baker, C. (2010). Development and application of Halanay-type
theory: Evolutionary differential and difference equations with time
lag, {\it Journal of Computational and Applied Mathematics}, 234,
2663-2682.

\item  
Baker, C. \& Tang, A. (1996). Generalized Halanay inequalities for
Volterra functional differential equations and discretized versions,
Invited plenary talk, in: Volterra Centennial Meeting, UTA
Arlington.


\item  
Chen, T. (2001). Global exponential stability of delayed Hopfield
neural networks,  {\it Neural Networks}, 14(8), 977-980.

\item 
Chen, T. \& Lu, W. (2003). Stability analysis of dynamical neural
networks, {\it IEEE International Conf. Neural Networks $\&$ Signal
Processing, Nanjing, China}, 1, 112-116.

\item  
Chen, T., Lu, W., \& Chen, G. (2005). Dynamical behaviors of a large
class of general delayed neural networks, {\it Neural Computation},
17, 949-968.

\item
Gil', M. (2013). {\it Stability of vector differential delay
equations}, Birkhauser.

\item  
Halanay, A. (1966). {\it Differential Equations}, New York: Academic
Press.


\item  
Liu, B., Lu, W., \& Chen, T. (2011a). Global almost sure self
synchronization of Hopfield neural networks with randomly switching
connetions, {\it Neural Networks}, 24, 305-310.

\item  
Liu, B., Lu, W., \& Chen, T. (2011b). Generalized Halanay inequality
and their applications to  neural networks with unbounded
time-varying delays, {\it IEEE Transactions on Neural Networks},
22(9), 1508-1513.

\item  
Liu, B., Lu, W., \& Chen, T. (2012). Stability analysis of some
delay differential inequalities with small time delays and its
applications, {\it Neural Networks}, 33, 1-6.


\item  
Lu, W., \&Chen, T. (2004). On periodic dynamical systems, {\it
Chinese Annals of Mathematics Series B}, 25(B:4), 455-462.

\item  
Mohamad, S., \& Gopalsamy, K. (2000). Continuous and discrete
Halanay-type inequalities, {\it Bulletin of Australian Mathematical
Society}, 61, 371-385.






\item  
Shen, Y., \& Wang, J. (2009). Almost sure exponential stability of
recurrent neural networks with Markovian switching, {\it IEEE
Transactions on Neural Networks}, 20(5), 840-855.

\item  
Wen, L., Yu, Y., \& Wang, W. (2008). Generalized Halanay
inequalities for dissipativity of Volterra functional differential
equations, {\it Journal of Mathematical Analysis and Applications},
347, 169-178.


\end{description}

\end{document}